\documentclass[12pt]{article}
\usepackage{graphicx}
\usepackage{amsfonts,graphicx, amsthm,amsmath,amssymb,amscd}
\usepackage[utf8]{inputenc}
\usepackage[russian]{babel}

\newtheorem{theorem}{Теорема}
\newtheorem*{collorary}{Следствие}
\newtheorem{proposition}{Предложение}
\newtheorem{lemma}{Лемма}
\newtheorem{conjecture}{Гипотеза}

\theoremstyle{definition}
\newtheorem{definition}{Определение}

\newtheorem{remark}{Замечание}

\begin{document}

\begin{center}
\begin{Large}
\textbf{Кратности предельных циклов, рождающихся при разрушении гиперболических полициклов}
\end{Large}

Дуков А.В.\footnote{Московский государственный университет им. М.В.Ломоносова, Россия}\footnote{Исследование выполнено при финансовой поддержке РФФИ в рамках научного проекта № 20-01-00420}

\today
\end{center}

\tableofcontents
\newpage

\begin{abstract}
We consider the multiplicity of limit cycles that appear when a hyperbolic polycycle is perturbed. We prove, in particular, that if such unfolding happens in generic finite-parameter families, the multiplicity of every new limit cycle does not exceed the number of separatrix connections in the polycycle.

-----

В статье рассматривается вопрос о кратности предельных циклов, рождающихся при разрушении произвольного гиперболического полицикла. В частности, доказано, что при возмущении внутри типичного конечно-параметрического семейства кратность любого родившегося предельного цикла не превосходит числа сепаратрисных связок, образующих полицикл.
\end{abstract}

\section{Введение}\label{sec:introduction}

Рассмотрим произвольное векторное поле $v_0$ на двумерном многообразии $\mathcal{M}$. На протяжении всей статьи многообразие $\mathcal{M}$ полагаем ориентируемым. Напомним определение полицикла:
\begin{definition}
\textit{Полициклом} векторного поля называется любой конечный ориентируемый граф $\Gamma$, удовлетворяющий следующим требованиям:
\begin{itemize}
\item вершинами графа $\Gamma$ являются особые точки поля;
\item рёбрами графа $\Gamma$ являются фазовые кривые поля, не являющиеся особыми точками; ориентация задаётся временем;
\item граф $\Gamma$ --- эйлеров (существует цикл, обходящий каждое ребро по одному разу).
\end{itemize}
\end{definition}

Пусть поле $v_0$ имеет некоторый полицикл $\gamma$. Рассмотрим $k$-параметрическое семейство $V=\{v_\delta\}$, $\delta\in B = (\mathbb{R}^k, 0)$, возмущающее данный полицикл. Будем говорить, что при разрушении полицикла $\gamma$ поля $v_0$ в семействе $V$ рождается предельный цикл (кратности $m$), если существует такая стремящаяся к нулю (которому соответствует поле $v_0$) последовательность значений параметров $\{\delta_\alpha\}_{\alpha \in \mathbb{N}}$, что для любого $\alpha$ поле $v_{\delta_\alpha}$ имеет предельный цикл $LC(\delta_\alpha)$ (кратности $m$), причём последовательность предельных циклов $LC(\delta_\alpha)$ при $\delta_\alpha \to 0$ стремится в метрике Хаусдорфа к полициклу $\gamma$ или к его непустой части.

Максимальное число предельных циклов, которые рождаются при разрушении полицикла $\gamma$ при переходе к близкому к $v_0$ полю семейства, называется \textit{цикличностью} данного полицикла. Полицикл называется \textit{элементарным}, если он образован лишь элементарными особыми точками, то есть только особыми точками с хотя бы одним ненулевым собственным значением. Максимальная цикличность, которую может иметь нетривиальный элементарный полицикл, возмущаемый в типичном $k$-параметрическом семействе, обозначается через $E(k)$ или же $E(n,k)$, где $n$ --- число особых точек, образующих полицикл.

В 30-х годах Андронов и Леонтович, а несколько позже и Хопф доказали равенство $E(1)=1$. С 1970-х по 1993гг. в результате трудов целого ряда математиков (Муртада, Руссари, Руссо, Дюмортье \cite{DRR}, Грозовский \cite{Gr}, Ройтенберг \cite{R}, Трифонов \cite{Trif}) было получено равенство $E(2)=2$. Подробнее об истории исследования цикличности полициклов коразмерности 1 и 2 см. \cite{DRR}. В 1997 году Трифоновым получена оценка $E(3) = 3$ \cite{Trif}.

На рубеже веков были предприняты попытки оценить цикличность элементарных полициклов для произвольного числа параметров $k$. В 1995 году Ильяшенко и Яковенко доказали, что для любого $k$ число $E(k)$ конечно \cite{IY}. В 2003 году в работе \cite{K} Калошиным получена оценка:
\begin{align*}
E(k) \leq 2^{25k^2}.
\end{align*}
Чуть позже в 2010 году Каледа и Щуров \cite{KS} доказали неравенство:
\begin{align*}
E(n,k) \leq C(n)k^{3n},
\end{align*}
где $C(n) = 2^{5n^2+20n}$, $n$ --- число вершин полицикла.

Как видно из приведённого обзора, ранние оценки были точны, но относились лишь к полициклам малой коразмерности. Поздние оценки, полученные Калошиным, Каледой и Щуровым, распространяются на произвольное число параметров, но при этом вряд ли являются точными. Это объясняется тем, что вопрос об оценке цикличности достаточно сложен.

В связи с этим возникает идея рассмотреть заведомо более простую задачу: \textit{какова максимальная кратность предельного цикла, рождающегося при разрушении полицикла в типичном конечно-параметрическом семействе?} Обратим внимание, что в работе речь пойдёт не об элементарных, а лишь о гиперболических полициклах, то есть образованными лишь гиперболическими сёдлами. Оказывается, что поставленная задача легко решается для произвольного числа параметров, причём оценка кратности зависит от количества сёдел полицикла не более чем линейно.

\paragraph{Основные результаты.}
Пусть поле $v_0$ содержит полицикл $\gamma$, образованный $n$ сепаратрисными связками гиперболических сёдел $S_1, \ldots, S_n$ (некоторые из сёдел могут совпадать). Обозначим характеристические числа сёдел $S_1, \ldots, S_n$ через $\lambda_1, \ldots, \lambda_n$ соответственно. Напомним, что \textit{характеристическим числом} седла называется модуль отношения собственных чисел, причём отрицательное стоит в числителе.

Основной результат работы представлен следующими двумя теоремами.
\begin{theorem}\label{th:main}
При возмущении полицикла $\gamma$ в типичном $n$-параметрическом семействе $V$ кратность любого рождающегося предельного цикла не превосходит $n$.
\end{theorem}

Условие типичности следующее: характеристические числа $\lambda_1, \ldots, \lambda_n$ сёдел, образующих полицикл $\gamma$, должны удовлетворять неравенству
\begin{align}
\mathcal{L}_n(\lambda_1, \ldots, \lambda_n) \neq 0, \label{eq:L_n_ineq}
\end{align}
где $\mathcal{L}_n$ --- некоторый нетривиальный многочлен. При заданном числе $n$ данный многочлен можно выбрать не зависящим ни от полицикла $\gamma$, ни от поля $v_0$.



Как будет видно далее в параграфе \ref{sec:multiple_cycles}, из наличия кратного предельного цикла следует, что некоторая полиномиальная система однородных уравнений, коэффициенты которой зависят от характеристических чисел $\lambda_1, \ldots, \lambda_n$, имеет нетривиальное решение. Забегая вперёд, отметим, что многочлен $\mathcal{L}_n$ выражается через результант данной полиномиальной системы.

\begin{collorary}
Пусть поле $v_0 \in Vect^\infty(\mathcal{M})$ имеет описанный выше полицикл $\gamma$. Пусть характеристические числа сёдел данного полицикла удовлетворяют неравенству (\ref{eq:L_n_ineq}). Тогда при возмущении поля $v_0$ в пространстве $Vect^\infty(\mathcal{M})$ из полицикла $\gamma$ не рождается предельный цикл кратности больше $n$.
\end{collorary}
\begin{proof}
Предположим, что существует последовательность полей $v_k \to v_0$, содержащих предельный цикл $LC(v_k)$ кратности $n+1$ или более, причём данный цикл стремится при $k\to \infty$ к полициклу $\gamma$. Тогда рассмотрим такое $n$-параметрическое семейство $V:\mathbb{R}^n \to Vect^\infty(\mathcal{M})$, что для любого натурального $k$ отрезок $\delta_1 \in [\frac{1}{k}, \frac{1}{k+1}]$ переходит в отрезок $[v_k, v_{k+1}]$, луч $\delta_1 \in [1, +\infty)$ отображается в поле $v_1$, а луч $\delta_1 \in (-\infty,0]$ --- в поле $v_0$. Зависимость от остальных параметров $\delta_2, \ldots, \delta_n$ полагаем фиктивной. Тогда семейство $V$ удовлетворяет описанному выше условию типичности. Таким образом, приходим к противоречию с теоремой \ref{th:main}.
\end{proof}

В случае полицикла малой коразмерности многочлен $\mathcal{L}_n$ можно выписать явно. Для этого нам потребуется определить несколько многочленов.

Для любого натурального $n$ обозначим через $\Lambda_n$ следующий многочлен от характеристических чисел $\lambda_1, \ldots, \lambda_n$:
\begin{align*}
\Lambda_n(\lambda_1, \ldots, \lambda_n) = \prod\limits_{I\neq (0, \ldots, 0)} (\lambda^I - 1),
\end{align*}
где $I=(i_1, \ldots, i_n)$ --- мультииндекс. Через $\lambda^I$ мы обозначили произведение $\lambda_1^{i_1}\ldots\lambda_n^{i_n}$. Для любого $j=1, \ldots, n$ компонента мультииндекса $i_j \in \{0, 1\}$ определяет, входит ли число $\lambda_j$ в произведение $\lambda^I$ или нет. Например, $\Lambda_2(\lambda_1, \lambda_2) = (\lambda_1-1)(\lambda_2-1)(\lambda_1\lambda_2-1)$.

Помимо этого, через $M(\lambda_1, \lambda_2, \lambda_3)$ обозначим следующий многочлен:
\begin{align*}
M(\lambda_1, \lambda_2, \lambda_3) = 4(\lambda_1\lambda_2\lambda_3 -1) - (\lambda_1-1)(\lambda_2-1)(\lambda_3-1).
\end{align*}

\begin{theorem}\label{th:small}
При $n=1,2,3,4$ в качестве многочлена $\mathcal{L}_n$, задающего условие типичности (\ref{eq:L_n_ineq}), можно взять следующие многочлены:
\begin{enumerate}
\item $\mathcal{L}_1(\lambda_1) = \Lambda_1(\lambda_1)$;
\item $\mathcal{L}_2(\lambda_1, \lambda_2) = \Lambda_2(\lambda_1, \lambda_2)$;
\item $\mathcal{L}_3(\lambda_1, \lambda_2, \lambda_3) = \Lambda_3(\lambda_1, \lambda_2, \lambda_3)$;
\item $\mathcal{L}_4(\lambda_1, \lambda_2, \lambda_3, \lambda_4) = \Lambda_4(\lambda_1, \lambda_2, \lambda_3, \lambda_4)\cdot $

$\cdot M(\lambda_1,\lambda_2,\lambda_3) M(\lambda_1,\lambda_2,\lambda_4) M(\lambda_1,\lambda_3,\lambda_4) M(\lambda_2,\lambda_3,\lambda_4).$
\end{enumerate}
\end{theorem}


Вся работа разбита на два раздела. Первый раздел целиком занимают промежуточные вычисления, основная задача которых --- свести вопрос о наличии кратного предельного цикла к наличию решения некоторой системы алгебраических уравнений. Во втором разделе применяются алгебраические методы, которые позволяют исследовать полученную систему. При помощи этих методов и доказываются обе теоремы.

\section{От векторных полей к многочленам}
\subsection{Отображения соответствия сёдел}\label{sec:saddle_maps}

На протяжении всей статьи полагаем, что выполнены условия теоремы \ref{th:main}. К каждой сепаратрисной связке полицикла $\gamma$ проведём $C^\infty$-гладкую трансверсаль: для любого $i=1, \ldots, n$ трансверсаль к связке сёдел $S_i$ и $S_{i+1}$ обозначим через $\Gamma_i$ (рис. \ref{fig:initial_polycycle}a). Также считаем, что для любого $i=1,\ldots,n$ трансверсаль $\Gamma_i$ не зависит от параметра $\delta$ и параметр $\delta$ выбирается настолько малым, что каждая из трансверсалей $\Gamma_i$ остаётся трансверсалной возмущённому полю $v_\delta$.

В поле $v_\delta$ рассмотрим произвольное седло $S_i(\delta)$ и две соседние к нему трансверсали $\Gamma_{i-1}$ и $\Gamma_i$ (полагаем, что $\Gamma_0 = \Gamma_n$). Точку пересечения трансверсали $\Gamma_{i-1}$ и входящей сепаратрисы седла $S_i(\delta)$ обозначим через $s_i(\delta)$ (от слова <<stable>>). Точку пересечения трансверсали $\Gamma_i$ и выходящей сепаратрисы седла $S_i(\delta)$ обозначим через $u_i(\delta)$ (от слова <<unstable>>).

Точка $s_i(\delta)$ делит трансверсаль $\Gamma_{i-1}$ на две полутрансверсали: фазовые кривые, начинающиеся с одной из полутрансверсалей, проходят вдоль седла $S_i(\delta)$ и либо покидают окрестность полицикла, либо пересекают трансверсаль, отличную от $\Gamma_i$ (здесь и далее речь идёт о первом пересечении с любой из трансверсалей); фазовые кривые, начинающиеся с другой полутрансверсали, проходят вдоль седла $S_i(\delta)$ и пересекают трансверсаль $\Gamma_i$. Последнюю из этих двух полутрансверсалей обозначим через $\Gamma_{i-1}^-(\delta)$.

Выберем на трансверсали $\Gamma_{i-1}$ такой натуральный параметр (карту), что точка $s_i(\delta)$ имеет в этой карте координату 0, а любая точка полутрансвесали $\Gamma_{i-1}^-(\delta)$ имеет положительную координату.

Точка $u_i(\delta)$ делит трансверсаль $\Gamma_i$ на две полутрансверсали. Ту из этих полутрансверсалей, которую пересекают фазовые кривые, берущие начало с полутрансверсали $\Gamma_{i-1}^-(\delta)$, обозначим через $\Gamma_i^+(\delta)$. Выберем на трансверсали $\Gamma_i$ такой натуральный параметр (карту), что точка $u_i(\delta)$ имеет в этой карте координату 0, а любая точка полутрансверсали $\Gamma_i^+(\delta)$ имеет положительную координату.

Таким образом, на каждой трансверсали $\Gamma_i$ нами выбраны две карты. Обозначим координату точки $u_i(\delta)$ в карте, соответствующей полутрансверсали $\Gamma_i^-(\delta)$, через $\tau_i(\delta)$. Тогда переход из карты, соответствующей полутрансверсали $\Gamma_i^+(\delta)$, в карту, соответствующую полутрансверсали $\Gamma_i^-(\delta)$, осуществляется отображением $x \mapsto \tau_i(\delta) \pm x$. 

Для любого $i=1,\ldots,n$ определено отображение соответствия $\Delta_i(\delta, \cdot): \Gamma_{i-1}^-(\delta) \to \Gamma_i^+(\delta)$ седла $S_i(\delta)$. В выбранных выше координатах на полутрансверсалях $\Gamma_{i-1}^-(\delta)$ и $\Gamma_i^+(\delta)$ отбражение $\Delta_i(\delta, \cdot)$ принимает вид $\Delta_i(\delta, \cdot): \mathbb{R}_{>0} \to \mathbb{R}_{>0}$. Это отображение непрерывно зависит от параметра $\delta$, а при фиксированном $\delta$ является $C^\infty$-гладким по аргументу $x$.

\textbf{\begin{figure}[h!]
\centering
\includegraphics[width=1\textwidth]{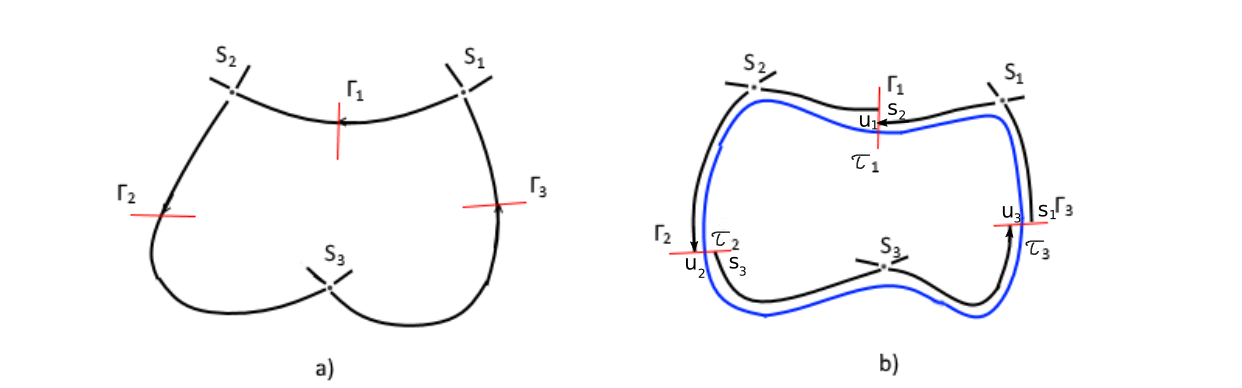}\\
\caption{a) Невозмущённый полицикл; b) предельный цикл, родившийся из полицикла.}\label{fig:initial_polycycle}
\end{figure}}

Рассмотрим отображение Пуанкаре $\Delta(\delta, \cdot): \Gamma_n^-(\delta) \to \Gamma_n$ заданного полицикла (рис. \ref{fig:initial_polycycle}b). Оно представимо в виде:
\begin{align}
\Delta(\delta, \cdot) = f_n(\delta, \cdot) \circ \ldots \circ f_1(\delta, \cdot), \label{eq:saddle_maps_Delta}
\end{align}
где
\begin{align}
f_i(\delta, \cdot): \Gamma_{i-1}^-(\delta) \to \Gamma_i, \quad f_i(x) = \tau_i(\delta) \pm \Delta_i(\delta, x). \label{eq:saddle_maps_fs}
\end{align}
Знак в формуле (\ref{eq:saddle_maps_fs}) определяется топологией полицикла. Точное его задание нам будет не важно.

\subsection{Уравнения на кратные предельные циклы}\label{sec:Poincare}
Пусть при некотором значении $\delta$ поле $v_\delta$ имеет предельный цикл $LC(\delta)$, родившийся при разрушении исходного полицикла $\gamma$. Пусть этот цикл пересекает полутрансверсаль $\Gamma_n^-(\delta)$ в точке с координатой $x_0 = x_0(\delta)$. Тогда отображение Пуанкаре имеет неподвижную точку, то есть пара $(\delta, x_0(\delta))$ является решением уравнения:
\begin{align}
\Delta(\delta, x) = x.\label{eq:Poincare_fixed}
\end{align}
Если же предельный цикл имеет кратность $n+1$ или более, то при $x=x_0(\delta)$ также выполнены следующие равенства:
\begin{eqnarray}
& \Delta'(\delta, x) = 1, \label{eq:Poincare_first_der} & \\
& \Delta^{(l+1)}(\delta, x) = 0, \label{eq:Poincare_ders} \quad l=1, \ldots, n-1. &
\end{eqnarray}

На протяжении всей статьи под производными $(\cdot)'$ и $(\cdot)^{(l)}$ будем понимать производную по аргументу $x$. Рассмотрим функцию 
\begin{align}
\mathcal{D}(\delta, x) = \ln\Delta'(\delta, x)
\end{align}
и связанную с ней следующую систему уравнений:
\begin{align}
\mathcal{D}^{(l)}(\delta, x) = 0, \quad l=0, \ldots, n-1. \label{eq:Poincare_another_form}
\end{align}
Поскольку векторное поле рассматривается на ориентируемом многообразии, то отображение Пуанкаре $\Delta(\delta, \cdot)$ есть сохраняющий ориентацию диффеоморфизм, заданный на полутрансверсали $\Gamma_n^-(\delta)$. Следовательно, его производная всегда положительна, что позволяет в определении функции $\mathcal{D}$ применить к функции $\Delta'$ логарифм.

Заметим, что если $x = x_0(\delta)$ --- неподвижная точка кратности $n+1$ функции $\Delta$, то $x_0(\delta)$ также удовлетворяет системе (\ref{eq:Poincare_another_form}). Это следует из того факта, что в малой окрестности точки $x_0(\delta)$ выполнено соотношение:
\begin{align*}
\mathcal{D}(\delta, x) = \ln \Delta'(\delta, x) = \ln \Big(1 + o\big((x-x_0)^{n-1}\big)\Big) = o\big((x-x_0)^{n-1}\big).
\end{align*}

\subsection{Общий вид производных большого порядка отображения Пуанкаре}
Предыдущий параграф наводит на мысль, что достаточно исследовать не само отображение Пуанкаре, а функцию $\mathcal{D}(\delta, x)$. Оказывается, что сколь угодно большие производные функции $\mathcal{D}(\delta, x)$ могут быть записаны в некоторой удобной форме.

Введём обозначение:
\begin{align}
F_i = f_i \circ \ldots \circ f_0, \quad f_0 = \mathrm{id}, \quad i=0, \ldots, n; \label{eq:common_view_Fs}
\end{align}
\begin{align}
Z_i = \frac{F_{i-1}'}{F_{i-1}}, \quad i=1,\ldots,n, \label{eq:limit_Zs}
\end{align}
где функции $f_i$ определяются формулой (\ref{eq:saddle_maps_fs}). Композиция берётся по аргументу $x$. В частности, $F_0(\delta, x) = x$ и $Z_1(\delta, x) = \frac{1}{x}$.

В новых обозначениях равенство $\mathcal{D}(\delta, x) = 0$ переписывается следующим образом:
\begin{align}
\mathcal{D}(\delta, x) = \sum\limits_{i=1}^n \ln |f_i'(F_{i-1})| = 0. \label{eq:common_view_D}
\end{align}
Помимо этого, введём ещё одно обозначение:
\begin{align}
\mu_{iq}(\delta, x) = y^q\frac{d^q}{dy^q}\ln |f_i'(y)| \Big|_{y = F_{i-1}(\delta, x)}, \quad i=1,\ldots,n; \; q \in \mathbb{N}. \label{eq:common_view_gamma}
\end{align}
В дальнейшем зависимость от переменой $x$ и параметра $\delta$ будем опускать и писать $\mu_{iq}$ и $F_{i-1}$.

\begin{proposition}\label{prop:common_view}
Для любого $l \in \mathbb{N}$ существует такой многочлен $P_{nl}$ c целыми коэффициентами, что $l$-я производная функции $\mathcal{D}$ (см. формулу (\ref{eq:Poincare_another_form})) имеет вид:
\begin{align}
\mathcal{D}^{(l)} = P_{nl}(\mu_{iq},Z_i), \quad i = 1 , \ldots, n, \quad q = 1, \ldots, l. \label{eq:common_view_terrible}
\end{align}
Многочлен $P_{nl}$ является однородным многочленом степени $l$ по переменным $Z_1, \ldots, Z_n$.
\end{proposition}
\begin{proof}
Докажем утверждение индукцией по $l$. \textit{База индукции}. При $l=1$ из соотношения (\ref{eq:common_view_D}) имеем:
\begin{align}
\mathcal{D}' = \sum\limits_{i=1}^n \frac{d}{dx} \ln |f_i'(F_{i-1})| = \sum\limits_{i=1}^n y\frac{d}{dy} \ln |f_i'(y)| \Big|_{y=F_{i-1}} \frac{F_{i-1}'}{F_{i-1}} = \sum\limits_{i=1}^n \mu_{i1} Z_i. \label{eq:common_view_l1}
\end{align}
Обозначим полученный многочлен через $P_{n1}(\mu_{i1}, Z_i)$, $i=1,\ldots,n$.

\textit{Шаг индукции}. Пусть для некоторого $l$ утверждение верно. Тогда получаем:
\begin{align}
\mathcal{D}^{(l+1)} = \frac{d}{dx} P_{nl}(\mu_{iq},Z_i) = \sum_{\substack{ 1 \leq i \leq n, \\ 1 \leq q \leq l}} \frac{\partial P_{nl}}{\partial \mu_{iq}} \mu_{iq}' + \sum\limits_{i=1}^n \frac{\partial P_{nl}}{\partial Z_i} Z_i'.\label{eq:common_view_prod_deriv}
\end{align}

С учётом обозначений (\ref{eq:limit_Zs}) и (\ref{eq:common_view_gamma}) производная функции $\mu_{iq}$ имеет вид:
\begin{align}
\mu_{iq}' = \left(q y^{q-1}\frac{d^q}{dy^q} + y^q\frac{d^{q+1}}{dy^{q+1}}\right)\ln |f_i'(y)| \Bigg|_{y = F_{i-1}}F_{i-1}' = (q\mu_{iq} + \mu_{i,q+1})Z_i. \label{eq:common_view_gamma_derive}
\end{align}

Чтобы вычислить производную $Z_i'$, найдём выражение $\frac{F_{i-1}''}{F_{i-1}'}$. Из равенств (\ref{eq:common_view_Fs}) и (\ref{eq:common_view_gamma}) имеем:
\begin{eqnarray}
& \frac{F_{i-1}''}{F_{i-1}'} = (\ln |F_{i-1}'|)' = \sum\limits_{j=1}^{i-1}(\ln |f_j'(F_{j-1})|)' = & \nonumber \\
& = \sum\limits_{j=1}^{i-1}y\frac{d}{dy}\ln |f_j'(F_{j-1}(y))|\Bigg|_{y = F_{j-1}}\frac{F_{j-1}'}{F_{j-1}} = \sum\limits_{j=1}^{i-1}\mu_{j1}Z_j.& \label{eq:common_view_F_over_F}
\end{eqnarray}
Пользуясь равенством (\ref{eq:limit_Zs}) и соотношением (\ref{eq:common_view_F_over_F}), найдём производную $Z_i'$:
\begin{align}
Z_i' = \left(\frac{F_{i-1}'}{F_{i-1}} \right)' = \frac{F_{i-1}''}{F_{i-1}} - \frac{F_{i-1}'^2}{F_{i-1}^2} = \frac{F_{i-1}''}{F_{i-1}'}Z_i - Z_i^2 = -Z_i^2 + Z_i\sum\limits_{j=1}^{i-1}\mu_{j1}Z_j.
\label{eq:common_view_low_tr}
\end{align}
Подставляем формулы (\ref{eq:common_view_gamma_derive}) и (\ref{eq:common_view_low_tr}) в выражение (\ref{eq:common_view_prod_deriv}). Получаем:
\begin{align}
\mathcal{D}^{(l+1)} = \sum_{\substack{ 1 \leq i \leq n, \\ 1 \leq q \leq l}} (q\mu_{iq} + \mu_{i,q+1})Z_i \frac{\partial P_{nl}}{\partial \mu_{iq}}  + \sum\limits_{i=1}^n (-Z_i + \sum\limits_{j=1}^{i-1}\mu_{j1}Z_j) Z_i \frac{\partial P_{nl}}{\partial Z_i}. \label{eq:common_view_res_deriv}
\end{align}
Нетрудно видеть, что полученное выражение есть однородный по переменным $Z_i$ многочлен степени $l$ с целыми коэффициэнтами, который мы обозначим через $P_{n,l+1}(\mu_{iq}, Z_i)$, где $i=1, \ldots, n$, $q = 1, \ldots, l+1$. Предложение \ref{prop:common_view} доказано.
\end{proof}

\subsection{Предельный переход при $\delta, x \to 0$. O-символика}\label{sec:saddle_limit}
В этом и следующих двух параграфах мы изучим предельные свойства производных отображения $\mathcal{D}$, фигурирующих в системе уравнений (\ref{eq:Poincare_another_form}). Оказывается, что предельные значения производных функции $\mathcal{D}$ при $\delta, x \to 0$ описываются некоторыми многочленами, зависящими лишь от характеристических чисел $\lambda_1, \ldots, \lambda_n$.

Чтобы найти предел при $\delta, x \to 0$ функции $\mu_{iq}(\delta,x)$, задаваемой формулой (\ref{eq:common_view_gamma}), нам потребуется следующая лемма:
\begin{lemma}\label{lmm:saddle_limit}
Рассмотрим конечно-параметрическое семейство $V = \{v_\delta\}$, $\delta \in (\mathbb{R}^k, 0)$ $C^\infty$-гладких векторных полей на двумерной плоскости. Пусть $\Delta_S(\delta, x)$ --- отображение соответствия гиперболического седла $S(\delta)$ поля $v_\delta$ с характеристическим числом $\lambda(\delta)$, $\lambda(0) = \lambda$. Тогда для любого натурального $q$ выполнено следующее соотношение:
\begin{align*}
\lim\limits_{\delta,x\to 0} x^q \frac{d^q}{dx^q} \ln\Delta_S^\prime(\delta, x) = (-1)^{q-1}(q-1)!(\lambda-1).
\end{align*}
\end{lemma}

Для доказательства этой леммы введём два полезных класса функций, предложенных Трифоновым в работе \cite{Trif}. Пусть $\lambda\in \mathbb{R}$ и функция $f(x) \in C^ r(\mathbb{R}, 0)$ непрерывно зависит от параметра $\delta$.
\begin{enumerate}
\item Будем говорить, что функция $f$ принадлежит классу $\tilde{o}_r^\lambda$, если для любого $m=0, \ldots, r$ имеет место предел:
\begin{align*}
\lim\limits_{x,\delta \to 0} x^{m-\lambda}f^{(m)}(x) = 0.
\end{align*}
\item Будем говорить, что функция $f(x) = f(\delta,x)$ принадлежит классу $\underset{\thicksim}{O}^\lambda_r$, если для любого малого $\varepsilon > 0$ функция $x^\varepsilon f(x)$ принадлежит классу $\tilde{o}_r^\lambda$.
\end{enumerate}
Как и в случае классов $o(1)$ и $O(1)$ мы будем вместо знака принадлежности классу использовать обычный знак равенства, например, $f(x) = x + \tilde{o}^\lambda_r$, подразумевая под этим, что $f(x) - x \in \tilde{o}_r^\lambda$.

Классы $\tilde{o}_r^\lambda$ и $\underset{\thicksim}{O}^\lambda_r$ являются более удобным инструментом, чем стандартные классы функций $o(x^\lambda)$ и $O(x^\lambda)$, поскольку допускают дифференцирование. Большой перечень их свойств приведён в работе $\cite{Trif}$, параграф 2.2. Нам же потребуются лишь следующие из них:

\begin{eqnarray}
& \forall \lambda, \mu \quad \underset{\thicksim}{O}^\lambda_r \cdot \underset{\thicksim}{O}^\mu_r = \underset{\thicksim}{O}^{\lambda + \mu}_r; \label{eq:Oo1} & \\
& \forall \lambda \quad (\underset{\thicksim}{O}^\lambda_r)' = \underset{\thicksim}{O}^{\lambda-1}_{r-1}; \label{eq:Oo2} & \\
& \forall \lambda > \mu > 0 \quad \underset{\thicksim}{O}^\lambda_r = \tilde{o}^\mu_r \to 0; \label{eq:Oo3} & \\
& \forall \lambda,\; \forall g\in C^r(\mathbb{R},0) \quad g\Big(\underset{\thicksim}{O}^\lambda_r \Big) = g(0) + \underset{\thicksim}{O}^\lambda_r; \label{eq:Oo4} & \\
& \forall \lambda \quad x^\lambda = \underset{\thicksim}{O}^\lambda_\infty. \label{eq:Oo5} &
\end{eqnarray}

В работе \cite{Trif} в параграфе 3.1 доказано, что отображение соответствия $\Delta_S$ из условия леммы \ref{lmm:saddle_limit} удовлетворяют соотношению:
\begin{align}
\Delta_S(x) = C(\delta)x^{\lambda(\delta)}(1+\underset{\thicksim}{O}^1_r) \quad \text{ при } x, \delta \to 0,\label{eq:trif_repr}
\end{align}
где $C(\delta)$, $\lambda(\delta)$ --- $C^1$-гладкие функции, причём $C(0) > 0$, $\lambda(0) = \lambda$, и $r$ --- некоторое сколь угодно большое натуральное число.

\begin{proof}[Доказательсто леммы \ref{lmm:saddle_limit}]
Пользуясь свойствами (\ref{eq:Oo1}), (\ref{eq:Oo2}) и (\ref{eq:Oo5}), продифференцируем равенство (\ref{eq:trif_repr}):
\begin{align*}
\Delta_S'(\delta,x) = C(\delta)\lambda(\delta) x^{\lambda(\delta)-1}(1+\underset{\thicksim}{O}^1_{r-1}).
\end{align*}
Логарифмируя полученное выражение, в силу свойства (\ref{eq:Oo4}) имеем:
\begin{align}
\ln\Delta_S'(\delta, x) = \ln C(\delta) + \ln \lambda(\delta) + (\lambda(\delta)-1)\ln x +\underset{\thicksim}{O}^1_{r-1}. \label{eq:Oo_ln_Delta}
\end{align}
Дифференцируем ещё $q$ раз. В силу соотношения (\ref{eq:Oo2}) получаем:
\begin{align*}
\frac{d^q}{dx^q}\ln\Delta_S'(\delta, x) = (-1)^{q-1}(q-1)!(\lambda(\delta)-1)x^{-q} + \underset{\thicksim}{O}^{1-q}_{r-q-1}.
\end{align*}
Домножим на $x^q$ и устремим $\delta, x \to 0$. Пользуясь свойствами (\ref{eq:Oo1}), (\ref{eq:Oo3}) и (\ref{eq:Oo5}), приходим к требуемому равенству. Лемма \ref{lmm:saddle_limit} доказана.
\end{proof}

Поскольку согласно определению \ref{eq:saddle_maps_fs} для функций $f_i$ выполнено равенство $|f_i'| = \Delta_i'$, то по лемме \ref{lmm:saddle_limit} существует предел:

\begin{align}
\lim\limits_{\delta,x\to 0} \mu_{iq}(\delta,x) = (-1)^{q-1}(q-1)!(\lambda_i-1), \label{eq:limit_gamma_limit}
\end{align}
где $\lambda_i$ --- характеристическое число седла $S_i$ поля $v_0$.

\subsection{Предельный переход при $\delta, x \to 0$. Производные отображения Пуанкаре}
Через $\mu_{iq}^0$ обозначим выражение в правой части равенства (\ref{eq:limit_gamma_limit}). Рассмотрим следующие многочлены:
\begin{align}
Q_{nl}(z_1, \ldots, z_n) = P_{nl}(\mu_{iq}^0, z_i), \quad i = 1, \ldots, n, \quad q = 1, \ldots, l, \label{eq:limit_Qs}
\end{align}
где $P_{nl}$ --- многочлены, существование которых утверждается в предложении \ref{prop:common_view}.

\begin{proposition}\label{prop:limit}
В условиях предложения \ref{prop:common_view} для любых натуральных чисел $n$ и $l$ многочлены $Q_{nl}$ удовлетворяют следующим свойствам:
\begin{enumerate}
\item $Q_{nl} (z_1, \ldots, z_n)$ --- однородный многочлен степени $l$;
\item $Q_{nl} (z_1, \ldots, z_n) \in \mathbb{Z}[\lambda_1,\ldots,\lambda_n][z_1, \ldots, z_n]$;
\item многочлены $Q_{nl}(z_1, \ldots, z_n)$ задаются рекуррентно следующим образом:
\begin{eqnarray}
& Q_{n,1} = \sum\limits_{i=1}^n (\lambda_i-1) z_i, \label{eq:reccurent_l1} & \\
& Q_{n,l+1} = \mathfrak{D}_n Q_{nl}, \label{eq:reccurent_Ps} &
\end{eqnarray}
где
\end{enumerate}
\begin{align}
\mathfrak{D}_n = (z_1, \ldots, z_n)
\left(\begin{array}{ccccc}
-1 & \lambda_1-1 & \lambda_1-1 & \ldots & \lambda_1-1 \\
0 & -1 & \lambda_2-1 & \ldots & \lambda_2-1 \\
 & & & & \\
\vdots &  & \ddots &  & \vdots \\
 & & & & \\
 0 & \ldots & \phantom{ddff} 0 & \phantom{df} -1 & \lambda_{n-1}-1 \\
 0 & \ldots &  & 0 & -1\phantom{df}
\end{array}\right)
\left(\begin{array}{c}
z_1\frac{\partial}{\partial z_1} \\
\\
\\
\vdots \\
\\
\\
z_n \frac{\partial}{\partial z_n}
\end{array}\right) \label{eq:reccurent_operator_D}
\end{align}
\end{proposition}
\begin{proof}
Докажем свойство 3. Свойства 1 и 2 следуют из свойства 3 очевидным образом.

Из соотношения (\ref{eq:limit_gamma_limit}) получаем, что для любого $i=1,\ldots,n$ и любого $q = 1, \ldots, l$ выполнены равенства:
\begin{align*}
\mu_{i,1}^0 = \lambda_i-1, \qquad q\mu_{iq}^0 + \mu_{i,q+1}^0 = 0.
\end{align*}
Таким образом, формула (\ref{eq:common_view_l1}) для многочлена $P_{n,1}$ влечёт формулу (\ref{eq:reccurent_l1}), а рекуррентное соотношение (\ref{eq:common_view_res_deriv}) для многочленов $P_{nl}$ влечёт рекуррентное соотношение
\begin{align}
Q_{n,l+1}(z_1, \ldots, z_n) = \sum\limits_{i=1}^n  \big(-z_i + \sum\limits_{j=1}^{i-1}(\lambda_i-1)z_j\big)z_i \frac{\partial Q_{nl}}{\partial z_i}. \label{eq:reccurent_Q}
\end{align}
Нетрудно видеть, что выражение справа от знака равенства в формуле (\ref{eq:reccurent_Q}) есть многочлен $Q_{nl}$, к которому применили оператор (\ref{eq:reccurent_operator_D}). Предложение доказано.
\end{proof}

В силу свойства 2 на многочлены $Q_{nl}$ можно смотреть как на многочлены $Q_{nl}(\lambda, z)$ от $2n$ переменных $\lambda = \lambda_1, \ldots, \lambda_n$, $z = z_1, \ldots, z_n$.

\begin{collorary}
Для любых натуральных $n$ и $l$, а также для любого $j=1,\ldots,n$ многочлены $Q_{nl}$ обладают следующим свойством:
\begin{eqnarray}
& Q_{nl}(\lambda, z) \Big|_{z_j=0} =
Q_{nl}(\lambda, z)\Big|_{\lambda_j=1} = Q_{n-1,l}(\lambda', z'), \label{eq:link_property} &
\end{eqnarray}
где $\lambda' = \lambda_1, \ldots, \hat{\lambda}_j, \ldots, \lambda_n$ и $z' = z_1, \ldots, \hat{z}_j, \ldots, z_n$.
\end{collorary}

Здесь через $\hat{\phantom{a}}$ обозначены отсутствующие в перечислении переменные. Доказательство легко следует из формул (\ref{eq:reccurent_l1} - \ref{eq:reccurent_operator_D}).

\subsection{Предельный переход при $\delta, x \to 0$. Кратные предельные циклы}\label{sec:multiple_cycles}

Поскольку и многочлены $P_{nl}$, и многочлены $Q_{nl}$ по переменным $z_1, \ldots, z_n$ являются однородными (см. предложения \ref{prop:common_view} и \ref{prop:limit}), то можно считать, что они заданы на проективном пространстве $\mathbb{R}P^{n-1}$. Точки пространства $\mathbb{R}P^{n-1}$ будем обозначать через $z = (z_1 : \ldots : z_n)$. Рассмотрим следующую функцию:
\begin{align}
\mathcal{Z}: (\delta, x) \mapsto \big(Z_1(\delta, x) : \ldots : Z_n(\delta, x)\big), \label{eq:cycle_Z}
\end{align}
где функции $Z_i(\delta,x)$ определяются формулой (\ref{eq:limit_Zs}).

В силу предложения \ref{prop:common_view} уравнения на $n+1$-кратный предельный цикл принимают вид:
\begin{eqnarray}
& \Delta(\delta, x) = x, \label{eq:cycle_Delta}  & \\
& \Delta'(\delta, x) = 1, \label{eq:cycle_Delta_derive}  & \\
& P_{nl}(\mu_{iq}(\delta, x), z) = 0, \quad l = 1, \ldots, n-1. \label{eq:cycle_Ps} &
\end{eqnarray}

Пусть существует последовательность $(\delta_\alpha, x_\alpha) \to 0$, соответствующая предельному циклу кратности $n+1$ или более. Тогда система (\ref{eq:cycle_Delta} - \ref{eq:cycle_Ps}) имеет решение $\delta_\alpha, x_\alpha, \mathcal{Z}(\delta_\alpha, x_\alpha)$. 

\begin{proposition}\label{prop:eqv_Q}
Пусть при возмущении полицикла $\gamma$ внутри семейства $V = \{v_\delta\}$, $\delta \in B = (\mathbb{R}^k,0)$ рождается предельный цикл кратности $n+1$ или более. Пусть $\{(\delta_\alpha, x_\alpha)\}_{\alpha=1}^\infty$ --- соответствующая этому циклу последовательность в пространстве $B \times (\mathbb{R}_{>0},0)$, причём при $(\delta_\alpha, x_\alpha) \to 0$ задаваемая формулой (\ref{eq:cycle_Z}) функция $\mathcal{Z}$ на этой последовательности стремится к некоторой точке $z \in \mathbb{R}P^{n-1}$. Тогда точка $z$ удовлетворяет следующей системе уравнений:
\begin{align}
Q_{nl}(z) = 0, \quad l=1, \ldots, n-1, \label{eq:cycle_Qs}
\end{align}
где многочлены $Q_{nl}$ задаются равенствами (\ref{eq:reccurent_l1} - \ref{eq:reccurent_operator_D}).
\end{proposition}

\begin{proof}
Утверждение очевидным образом следует из равенства (\ref{eq:cycle_Ps}) и определения многочленов $Q_{nl}$ (см. формулу (\ref{eq:limit_Qs})).
\end{proof}

Предложение \ref{prop:eqv_Q} можно сформулировать в более общем виде, который может оказаться полезным.
\begin{proposition}\label{prop:eqv_Q_common}
Пусть при возмущении полицикла $\gamma$ внутри семейства $V = \{v_\delta\}$, $\delta \in B = (\mathbb{R}^k,0)$ рождается предельный цикл кратности $m+2$ или более. Пусть $\{(\delta_\alpha, x_\alpha)\}_{\alpha=1}^\infty$ --- соответствующая этому циклу последовательность в пространстве $B \times (\mathbb{R}_{>0},0)$, причём при $(\delta_\alpha, x_\alpha) \to 0$ задаваемая формулой (\ref{eq:cycle_Z}) функция $\mathcal{Z}$ на этой последовательности стремится к некоторой точке $z \in \mathbb{R}P^{n-1}$. Тогда точка $z$ удовлетворяет следующей системе уравнений:
\begin{align}
Q_{nl}(z) = 0, \quad l=1, \ldots, m, \label{eq:cycle_Qs}
\end{align}
где многочлены $Q_{nl}$ задаются равенствами (\ref{eq:reccurent_l1} - \ref{eq:reccurent_operator_D}).
\end{proposition}
Доказательство аналогично доказательству предложения \ref{prop:eqv_Q}.

\begin{lemma}\label{lmm:cycle_CP}
Пусть поле $v_0$ имеет полицикл $\gamma$, образованный гиперболическими сёдлами $S_1, \ldots, S_n$, $n \geq 2$, с характеристическими числами $\lambda_1, \ldots, \lambda_n$, которые удовлетворяют неравенству $\lambda_1 \ldots \lambda_n \neq 1$. Пусть семейство $V=\{v_\delta\}$ возмущает поле $v_0$. Обозначим через $C$ множество таких пар $(\delta, x)$, что поле $v_\delta$ имеет имеет предельный цикл кратности как минимум 2, проходящий через точку с координатой $x$. Пусть $\mathfrak{Z} = \{ z \in \mathbb{R}P^{n-1} | \exists \{(\delta_\alpha, x_\alpha)\}_{\alpha=1}^\infty \subset C, \mathcal{Z}(\delta_\alpha, x_\alpha) \to z \text{ при } \alpha \to \infty\}$, где отображение $\mathcal{Z}$ задаётся формулой (\ref{eq:cycle_Z}).

Тогда $\mathfrak{Z} \subset \cup_{j=1}^n \mathbb{C}P^{n-2}_j$, где $\mathbb{C}P^{n-2}_j = \{ z = (z_1 : \ldots : z_n) \in \mathbb{R}P^{n-1} | z_j = 0 \}$.

\end{lemma}
\begin{proof}
Предположим, что существует такая точка $z = (z_1 : \ldots : z_n) \in \mathfrak{Z}$, что для любого $i = 1, \ldots, n$ координата $z_j$ отлична от нуля. Тогда существует такая последовательность $(\delta_\alpha, x_\alpha) \to 0$, соответствующая предельному циклу кратности два или более, что на этой последовательности функция $\mathcal{Z}$ стремится к точке $z$.

Докажем индукцией по $i=0, \ldots, n$, что выполнено следующее соотношение:
\begin{align}
F_i(\delta_\alpha, x_\alpha) = x^{\lambda_1(\delta_\alpha) \ldots \lambda_i(\delta_\alpha)} * \quad \text{ при } \delta_\alpha, x_\alpha \to 0, \label{eq:cycle_ln}
\end{align}
где функции $F_i$ определяются формулой (\ref{eq:common_view_Fs}). Здесь и далее $*$ означает умножение на некоторую функцию, отделённую от нуля и бесконечности.

База индукции $i=0$ очевидна: $F_0(\delta_\alpha, x_\alpha) = x_\alpha$. Для произвольной функции $g(\delta, x)$ под $g\big|_{(\delta_\alpha, x_\alpha)}$ будем подразумевать, что функция берётся в точке $(\delta_\alpha, x_\alpha)$.

Пусть утверждение выполнено для некоторого $i$. Поскольку мы предположили, что для любого $i=1, \ldots, n$ компоненты $z_i$ точки $z$ ненулевые, то имеем $\left.\frac{Z_{i+1}}{Z_i}\right|_{(\delta_\alpha, x_\alpha)} = *$. С другой стороны, из формулы (\ref{eq:limit_Zs}) следует, что
\begin{align*}
\left.\frac{Z_{i+1}}{Z_i}\right|_{(\delta_\alpha, x_\alpha)} = \left.\frac{f_i'(F_{i-1})F_{i-1}}{F_i}\right|_{(\delta_\alpha, x_\alpha)}.
\end{align*}
Выразив из этого равенства $F_i$, приходим к соотношению:
\begin{align}
F_i\Big|_{(\delta_\alpha, x_\alpha)} = f_i'(F_{i-1}) F_{i-1}*\Big|_{(\delta_\alpha, x_\alpha)} \quad \text{ при } \delta_\alpha, x_\alpha \to 0. \label{eq:cycle_lnF}
\end{align}
Из формул (\ref{eq:saddle_maps_fs}) и (\ref{eq:Oo_ln_Delta}) следует, что
\begin{align}
f_i'(\delta, x) = \Delta_i'(\delta, x) = x^{\lambda_i(\delta) - 1} * \quad \text{ при } \delta, x \to 0. \label{eq:cycle_f_diff}
\end{align}
Подставляя это выражение в равенство (\ref{eq:cycle_lnF}), получаем рекуррентное соотношение $F_i(\delta_\alpha, x_\alpha) = F_{i-1}(\delta_\alpha, x_\alpha)^{\lambda_i(\delta_\alpha)}$, что влечёт формулу (\ref{eq:cycle_ln}).

Снова индукцией по $i=1, \ldots, n$ докажим следующую формулу:
\begin{align}
F_i'(\delta_\alpha, x_\alpha) = x^{\lambda_1(\delta_\alpha)\ldots\lambda_i(\delta_\alpha)-1}*. \label{eq:cycle_F_diff}
\end{align}
База индукции $i=1$ следует из равенства (\ref{eq:cycle_f_diff}). Также из равенства (\ref{eq:cycle_f_diff}) и доказанного соотношения (\ref{eq:cycle_ln}) получаем шаг индукции:
\begin{align*}
F_i'(\delta_\alpha, x_\alpha) = f_n'(F_{i-1}) F_{i-1}' \Big|_{(\delta_\alpha, x_\alpha)} = \left(x^{\lambda_1(\delta_\alpha)\ldots\lambda_{i-1}(\delta_\alpha)}\right)^{\lambda_i(\delta_i)-1} x^{\lambda_1(\delta_\alpha)\ldots\lambda_{i-1}(\delta_\alpha)-1}*.
\end{align*}
Поскольку для любого $\alpha$ пара $(\delta_\alpha, x_\alpha)$ соответсвует предельному циклу кратности как минимум два, то выполнено равенство $\Delta'(\delta_\alpha, x_\alpha) = F_n'(\delta_\alpha, x_\alpha) = 1$. Взяв от этого равенства логарифм и примененив формулу (\ref{eq:cycle_F_diff}), приходим к соотношению:
\begin{align*}
\ln \Delta'(\delta_\alpha, x_\alpha) = (\lambda_1(\delta_\alpha)\ldots\lambda_n(\delta_\alpha) - 1)\ln x_\alpha + O(1) = 0.
\end{align*}
После деления на $\ln x_\alpha$ и взятия предела при $\alpha \to \infty$ получаем равенство $\lambda_1 \ldots \lambda_n = 1$, что противоречит условию. Следовательно, хотя бы одна из координат точки $z$ равна нулю.

\end{proof}

\subsection{План доказательства теоремы \ref{th:main}} \label{sec:th2}
\begin{proof}
Рассмотрим случай $n=1$: полицикл образован одним седлом с характеристическим числом $\lambda_1$, то есть является петлёй сепаратрисы седла $S_1$. Пусть число $\lambda_1$ не корень многочлена $\mathcal{L}_1(\lambda_1) = \lambda_1-1$. Предположим, что при возмущении полицикла в семействе $V$ рождается двукратный предельный цикл, то есть цикл, для которого отображение Пуанкаре $\Delta(x)$ в соответствующей точке удовлетворяет уравнениям: $\Delta(x) = x$, $\Delta'(x) = 1$. В частности, из последнего равенства следует, что $\ln \Delta'(x) = 0$. Но в силу соотношения (\ref{eq:Oo_ln_Delta}) при $\lambda_1 \neq 1$ имеем $\ln \Delta'(x) \to \pm \infty$ при $\delta, x \to 0$, что приводит к противоречию. Следовательно, для $n=1$ теорема доказана.

Перейдём к случаю $n \geq 2$. Предположим, что в семействе $V$ рождается предельный цикл кратности $n+1$. Тогда предложение \ref{prop:eqv_Q} влечёт, что заданная на проективном пространстве $\mathbb{R}P^{n-1}$ система однородных уравнений (\ref{eq:cycle_Qs}) имеет хотя бы одно решение.

Более того, пусть выполнено следующее неравенство на характеристические числа:
\begin{align}
\lambda_1\ldots\lambda_n \neq 1. \label{eq:th2_lambdas_prod}
\end{align}
Тогда из леммы \ref{lmm:cycle_CP} следует, что для некоторого $j = 1, \ldots, n$ система
\begin{align}
Q_{nl}(z) = 0, \quad l = 1, \ldots, n-1 \label{eq:th2_QZ}
\end{align}
имеет решение в подпространстве $\mathbb{R}P^{n-2}_j$. Согласно свойству (\ref{eq:link_property}) для некоторого $j=1, \ldots, n$ следующая система уравнений имеет нетривиальное действительное решение:
\begin{align*}
Q_{n-1,l}(\lambda_1, \ldots, \hat{\lambda}_j, \ldots, \lambda_n,z_1,\ldots, \hat{z}_j, \ldots, z_n) = 0, \quad l = 1, \ldots, n-1.
\end{align*}
Здесь через $\hat{\phantom{a}}$ мы снова обозначили отсутствующие в перечислении переменные. Таким образом, для любых натуральных чисел $n$ и $l$ многочлен $Q_{nl}$ зависит от $2(n-1)$ переменных. Обозначим их через $\mu_1, \ldots, \mu_{n-1}$ и $w_1, \ldots, w_{n-1}$. Рассмотрим систему:
\begin{align}
Q_{n-1,l}(\mu_1, \ldots, \mu_{n-1}, w_1, \ldots, w_{n-1}) = 0, \quad l = 1, \ldots, n-1. \label{eq:Q_mu_W}
\end{align}
При фиксированных значения переменных $\mu_1, \ldots, \mu_{n-1}$ мы имеем систему из $n-1$ уравнения на $n-2$-мерном проективном пространстве. Она имеет нетривиальное (вообще говоря, комплексное) решение тогда и только тогда, когда её результант $\mathcal{R}_{n-1}(\mu_1, \ldots, \mu_{n-1})$ равен нулю \cite{Eis}.

Нетривиальность результанта системы (\ref{eq:Q_mu_W}) следует из леммы:
\begin{lemma}\label{lmm:nonull}
Для любого натурального $n\geq 2$ многочлен $\mathcal{R}_{n-1}(\mu, \ldots, \mu)$ от действительной переменной $\mu$ не равен тождественно нулю.
\end{lemma}
Данная лемма будет доказана в параграфе \ref{sec:nonull}. Рассмотрим следующий многочлен:
\begin{align}
\mathcal{L}_n(\lambda_1, \ldots, \lambda_n) = (\lambda_1\ldots\lambda_n-1)\prod\limits_{j=1}^n \mathcal{R}_{n-1}(\lambda_1, \ldots, \hat{\lambda}_j, \ldots, \lambda_n). \label{eq:main_Q}
\end{align}
Пусть характеристические числа $\lambda_1, \ldots, \lambda_n$ таковы, что значение многочлена $\mathcal{L}_n(\lambda_1, \ldots, \lambda_n)$ не равно нулю. Тогда выполнено необходимое для леммы \ref{lmm:cycle_CP} неравенство (\ref{eq:th2_lambdas_prod}) и для любого $j = 1, \ldots, n$ система (\ref{eq:th2_QZ}) не имеет решения. Значит, предельного цикла кратности $n+1$ в семействе $V$ родиться не может. Приходим к противоречию. Следовательно, многочлен $\mathcal{L}_n$ искомый.

Для завершения доказательства осталось проверить, что неравенство $\mathcal{L}_n(\lambda_1, \ldots, \lambda_n)\neq 0$ является условием типичности, накладываемое на исходное поле $v_0$.

По условию невозмущённый полицикл $\gamma$ исходного векторного поля $v_0$ образован сёдлами $S_1, \ldots, S_n$ с характеристическими числами $\lambda_1, \ldots, \lambda_n$ соответственно, причём некоторые из сёдел могут совпадать. Если никакие два седла не совпадают, то все характеристические числа являются независимыми величинами, принимающими произвольные положительные значения. Поскольку из леммы \ref{lmm:nonull} следует, что многочлен $\mathcal{L}_n$ не тождественно равен нулю, то множество значений характеристических чисел, задаваемых неравенством (\ref{eq:L_n_ineq}), открыто и всюду плотно в $\mathbb{R}_{>0}^n$.

Предположим, что некоторые из сёдел совпали. Значит, совпали и их характеристические числа. Но в силу леммы \ref{lmm:nonull} снова получаем, что результант $\mathcal{R}_{n-1}$ системы (\ref{eq:th2_QZ}) при отождествлении некоторых характеристических чисел тоже остаётся нетривиальным. Следовательно, неравенство (\ref{eq:L_n_ineq}) задаёт условие типичности.

Таким образом, без учёта леммы \ref{lmm:nonull} теорема \ref{th:main} доказана.
\end{proof}

\section{Нетривиальность результанта}\label{sec:nonull}
В этом параграфе мы докажем лемму \ref{lmm:nonull}, что завершит доказательство теоремы \ref{th:main}.
\begin{proof}
От противного. Пусть многочлен $\mathcal{R}_{n-1}(\mu, \ldots, \mu)$ тождественно равен нулю. Тогда в силу равенства (\ref{eq:reccurent_l1}) многочлен $Q_{n-1,1} = \\ = Q_{1,n-1}(w_1, \ldots, w_{n-1})$ имеет вид:
\begin{align*}
Q_{n-1,1}(w_1, \ldots, w_{n-1}) = (\mu - 1) (w_1 + \ldots + w_{n-1}).
\end{align*}
Так как порождающий многочлены $Q_{n-1,l}$ оператор $\mathfrak{D}_{n-1}$ линеен (см. формулу (\ref{eq:reccurent_operator_D})), то каждый многочлен $P_{n-1,l}$ можно сократить на множитель $\mu-1$.

Устремим $\mu$ к единице. Тогда согласно формуле (\ref{eq:reccurent_operator_D}) оператор $\mathfrak{D}_{n-1} = \mathfrak{D}_{n-1}(\mu, w_1, \ldots, w_{n-1})$ в пределе перейдёт в оператор $\overline{\mathfrak{D}}_{n-1} = - ( w_1^2\frac{\partial}{\partial w_1} + \ldots + w_{n-1}^2\frac{\partial}{\partial w_{n-1}})$. Следовательно, после деления на $\mu-1$ и перехода к пределу $\mu \to 1$ многочлены $\frac{1}{\mu-1}Q_{n-1,l}$ перейдут в следующие многочлены:
\begin{eqnarray*}
& \overline{Q}_{n-1,1} = \frac{1}{\mu-1} Q_{n-1,1} = w_1 + \ldots + w_{n-1} & \\
& \overline{Q}_{n-1,l} = \overline{\mathfrak{D}}_{n-1}^{l-1} \overline{Q}_{n-1,1} = \overline{\mathfrak{D}}_{n-1}^{l-1} (w_1 +\ldots + w_{n-1}) = & \\
& = (-1)^{l-1} (l-1)! (w_1^l + \ldots + w_{n-1}^l), \quad l = 1, \ldots, n-1. &
\end{eqnarray*}

Из нашего предположения и свойства результанта следует, что для любого $\mu \neq 1$ система уравнений (\ref{eq:Q_mu_W}) имеет на пространстве $\mathbb{C}P^{n-2}$ хотя бы одно решение, которое мы обозначим через $W(\mu)$. Поскольку проективное пространство $\mathbb{C}P^{n-2}$ --- компакт, то существует (не обязательно единственная) предельная точка $W(1) \in \mathbb{C}P^{n-2}$, к которой накапливаются точки $W(\mu)$ при $\mu \to 1$.

Так как многочлены $\frac{1}{\mu-1} Q_{n-1,l}$ непрерывно зависят от параметра $\mu$, то для любого $l=1,\ldots,n-1$ многочлен $\overline{Q}_{nl}$ зануляется в точке $W(1)$. Таким образом, чтобы прийти к противоречию, нам достаточно показать, что система из симметрических многочленов
\begin{align}
p_l(w_1, \ldots, w_{n-1}) = w_1^l + \ldots + w_{n-1}^l = 0, \quad l=1, \ldots, n-1, \label{eq:nonull_ps}
\end{align}
не имеет решения в пространстве $\mathbb{C}P^{n-2}$. Отсутствие нетривиальных комплексных решений у системы (\ref{eq:nonull_ps}) --- хорошо известный факт. Докажем его для полноты.

Рассмотрим следующие симметрические многочлены:
\begin{eqnarray}
& \sigma_0 (w_1, \ldots, w_{n-1}) = 1; \nonumber & \\
& \sigma_l (w_1, \ldots, w_{n-1}) = \sum\limits_{1\leq i_1 < \ldots < i_l\leq n-1} w_{i_1} \ldots w_{i_l}, \quad l=1, \ldots, n-1. & \label{eq:nonull_es}
\end{eqnarray}
Как известно, симметрические многочлены выражаются друг через друга. В частности, многочлены $p_l$ и $\sigma_l$ связывает тождество Ньютона (\cite{Pr}, параграф 11.1):
\begin{align*}
l \sigma_l = \sum \limits_{i=1}^l (-1)^{i-1} \sigma_{l-i} \; p_i, \quad l=1, \ldots, n-1.
\end{align*}
Отcюда в силу соотношения (\ref{eq:nonull_ps}) приходим к системе:
\begin{align}
\sigma_l = 0, \quad l=1, \ldots, n-1. \label{eq:nonull_e_system}
\end{align}
Заметим, что если бы у системы (\ref{eq:nonull_e_system}) имелось хотя бы одно нетривиальное решение $(w_1, \ldots, w_{n-1})$, то согласно теореме Виета многочлен
\begin{align*}
\prod\limits_{l=1}^{n-1}(w-w_l) = \sum\limits_{l=0}^{n-1} (-1)^l \sigma_l w^{n-l-1} = w^{n-1}
\end{align*}
имел бы хотя бы один отличный от нуля корень. Но это неверно, что приводит нас к противоречию. Следовательно, многочлен $\mathcal{R}_{n-1}(\mu, \ldots, \mu)$ нетривиален.

\end{proof}

Доказательство теоремы \ref{th:main} немедленно следует из этой леммы (см. параграф \ref{sec:th2}).

\section{Доказательство теоремы \ref{th:small}}\label{sec:small}
\begin{proof} \textbf{Случай} $n=1$ рассмотрен в доказательстве теоремы \ref{th:main}. Имеем:
\begin{align*}
\mathcal{L}_1(\lambda_1) = \lambda_1-1 = \Lambda_1(\lambda_1).
\end{align*}

При $n=2$, $3$ и $4$ план доказательства следующий. Сначала находим результант $\mathcal{R}_{n-1}$, непосредственно решая систему (\ref{eq:Q_mu_W}). Искомый многочлен $\mathcal{L}_n$ будет выражаться через результант $\mathcal{R}_{n-1}$ посредством формулы (\ref{eq:main_Q}).

\paragraph{Случай $n=2$.} Согласно равенству (\ref{eq:reccurent_l1}) в этом случае система (\ref{eq:Q_mu_W}) состоит из одного уравнения $(\mu_1 - 1)w_1 = 0$, которое имеет решение $w_1 \neq 0$ тогда и только тогда, когда равен нулю многочлен
\begin{align}
\mathcal{R}_1(\mu_1) = \mu_1 - 1. \label{eq:small_P1}
\end{align}
Из формулы (\ref{eq:main_Q}) имеем:
\begin{align*}
\mathcal{L}_2(\mu_1, \mu_2) = (\mu_1\mu_2-1)(\mu_1-1)(\mu_2-1) = \Lambda_1(\mu_1, \mu_2).
\end{align*}

\paragraph{Случай $n=3$.} Воспользовавшись рекуррентным соотношением (\ref{eq:reccurent_Ps}), получаем, что система (\ref{eq:Q_mu_W}) в этому случае имеет вид:
\begin{eqnarray}
& Q_{21}(w_1,w_2) = (\mu_1-1)w_1 + (\mu_2-1)w_2 = 0 \label{eq:small_n2_1} & \\
& Q_{22}(w_1,w_2) = -(\mu_1-1)w_1^2 + (\mu_1-1)(\mu_2-1)w_1w_2 - (\mu_2-1)w_2^2 = 0 \label{eq:small_n2_2} &.
\end{eqnarray}
Рассмотрим следующую линейную комбинацию:
\begin{align}
Q_{22}(w_1, w_2) + (w_1 + w_2)Q_{21}(w_1, w_2) = (\mu_1\mu_2 - 1)w_1w_2 = 0. \label{eq:small_n2_3}
\end{align}
Очевидно, что система (\ref{eq:small_n2_1}, \ref{eq:small_n2_3}) эквивалентна системе (\ref{eq:small_n2_1} - \ref{eq:small_n2_2}). Заметим, что система (\ref{eq:small_n2_1}, \ref{eq:small_n2_3}) имеет нетривиальное решение тогда и только тогда, когда равен нулю многочлен
\begin{align}
\mathcal{R}_2(\mu_1, \mu_2) = (\mu_1\mu_2-1)(\mu_1-1)(\mu_2-1). \label{eq:small_P2}
\end{align}
Из формулы (\ref{eq:main_Q}) находим:
\begin{align*}
\mathcal{L}_3(\mu_1, \mu_2, \mu_3) = (\mu_1\mu_2\mu_3-1)(\mu_1\mu_2-1)(\mu_1\mu_3-1)(\mu_2\mu_3-1)\cdot \\ \cdot(\mu_1-1)(\mu_2-1)(\mu_3-1) = \Lambda_3(\mu_1, \mu_2, \mu_3).
\end{align*}

\paragraph{Случай $n=4$.} Для удобства чтения мы будем опускать некоторые выкладки. Читателю не составит труда проверить их самостоятельно. Рассмотрим следующие три многочлена:
\begin{eqnarray}
& \tilde{Q}_{31}(w_1,w_2,w_3) = Q_{31}(w_1,w_2,w_3) = & \nonumber \\
& = (\mu_1-1)w_1 + (\mu_2-1)w_2 + (\mu_3-1)w_3 \label{eq:small_n3_1} & \\
& \nonumber & \\
& \tilde{Q}_{32}(w_1,w_2,w_3) = Q_{32}(w_1,w_2,w_3) + (w_1+w_2+w_3) Q_{31}(w_1,w_2,w_3) = \nonumber & \\
& = (\mu_1 \mu_2 -1) w_1 w_2 + (\mu_1 \mu_3 -1) w_1 w_3 + (\mu_2 \mu_3 -1) w_2 w_3 \label{eq:small_n3_2} & \\
& \nonumber & \\
& \tilde{Q}_{33}(w_1,w_2,w_3) = \mathcal{D}_3 \tilde{Q}_{32}(w_1,w_2,w_3) = \nonumber & \\
& = w_1\Big(-w_1+w_2(\mu_2-1)+w_3(\mu_3-1)\Big) (\mu_1\mu_2w_2+ \mu_1\mu_3w_3-w_2-w_3) - \nonumber & \\
& - w_2\Big(w_2-w_3(\mu_3-1)(\mu_1\mu_2 w_1 +\mu_2 \mu_3 w_3 -w_1 - w_3)\Big) + \label{eq:small_n3_3} & \\
& + w_3^2(-\mu_1\mu_3 w_1 - \mu_2\mu_3w_2 + w_1 + w_2) \nonumber &
\end{eqnarray}
Нетрудно видеть, что многочлены $\tilde{Q}_{31}$, $\tilde{Q}_{32}$ и $\tilde{Q}_{33}$ в кольце $\mathbb{Z}[\mu_1,\mu_2, \mu_3, w_1,w_2,w_3]$ порождают тот же идеал, что и многочлены $Q_{31}$, $Q_{32}$ и $Q_{33}$, задаваемые рекуррентной формулой (\ref{eq:limit_Qs}). Найдём, при каких $\mu_1$, $\mu_2$, $\mu_3$ система, образованная многочленами (\ref{eq:small_n3_1} - \ref{eq:small_n3_3}) имеет нетривиальное решение. Сделаем следующее преобразование:

\begin{eqnarray}
& \tilde{Q}_{33} - \Big(-w_1+(\mu_3-2)w_2 + (2\mu_3-3)w_3\Big)\tilde{Q}_{32} + (\mu_2\mu_3-1)\tilde{Q}_{31} = \nonumber & \\
& = w_1 w_3 \Big(w_2(\mu_1\mu_2\mu_3 + \mu_1\mu_2 + \mu_1\mu_3 - \mu_1-2) - w_3(\mu_1\mu_3-1)(\mu_3-1)\Big) = \nonumber & \\
& = w_1 w_3 L(w_2, w_3). \label{eq:small_n3_big} &
\end{eqnarray}

\begin{remark}\label{rem:Z_not_0}
Можем считать, что все три переменные $w_1$, $w_2$ и $w_3$ не равны нулю. Действительно, если какая-то из этих переменных равна нулю, то в силу свойства (\ref{eq:link_property}) исходная система из трёх уравнений $Q_{31}$, $Q_{32}$ и $Q_{33}$ превращается систему из двух многочленов $Q_{21}$ и $Q_{22}$ от двух других переменных. Следовательно, чтобы найти результант $\mathcal{R}_3$, ответ, полученный из предположения о неравенстве нулю указанных переменных, в конечном счёте следует домножить на многочлены $\mathcal{R}_2$ от всевозможных пар переменных $\mu_1, \mu_2, \mu_3$.
\end{remark}

Рассмотрим систему, образованную многочленами $\tilde{P}_{31}$, $\tilde{P}_{32}$ и $L$, приравненными к нулю. Пользуясь линейностью многочлена $L$, исключим из полученной системы переменную $w_3$. Приходим к следующей системе из двух уравнений:
\begin{eqnarray}
& w_1(\mu_1-1)(\mu_3-1)(\mu_1\mu_3-1) + \nonumber & \\
& + w_2(\mu_3-1)(2\mu_1\mu_2\mu_3 + \mu_1\mu_2 - \mu_1 -\mu_2-1) = 0 & \\
& \nonumber & \\
& w_2 \Big(w_1(\mu_1\mu_3-1)(2\mu_1\mu_2\mu_3 + \mu_1\mu_3 -\mu_1-\mu_3-1) + \nonumber & \\
& + w_2(\mu_2\mu_3-1)(\mu_1\mu_2\mu_3+\mu_1\mu_2 + \mu_1\mu_3 - \mu_1 -2)\Big) = 0 &
\end{eqnarray}

С учётом замечания \ref{rem:Z_not_0} разделим второе уравнение на $w_2$, перейдя таким образом к линейной системе. Полученная линейная система имеет нетривиальное решение тогда и только тогда, когда её определитель равен нулю. Выпишем этот определитель:

\begin{align*}
(\mu_3-1)(\mu_1\mu_3-1)(\mu_1\mu_2\mu_3-1)
\Big(4(\mu_1\mu_2\mu_3-1) - (\mu_1-1)(\mu_2-1)(\mu_3-1)\Big)
\end{align*}

В силу замечания \ref{rem:Z_not_0} данный многочлен следует домножить на результанты $\mathcal{R}_2(\mu_1, \mu_2)$, $\mathcal{R}_2(\mu_1, \mu_3)$ и $\mathcal{R}_2(\mu_2, \mu_3)$, найденные ранее (формула (\ref{eq:small_P2})). Полученный многочлен обозначим через $R^*$. Если исходная система, составленная из многочленов $Q_{31}$, $Q_{32}$ и $Q_{33}$ имеет нетривиальное решение при некоторых значениях $\mu_1,\mu_2,\mu_3$, то многочлен $R^*$ равен нулю. Следовательно, его множество нулей содержит в качестве подмножества нули искомого результанта $\mathcal{R}_3$.

На самом деле, результант $\mathcal{R}_3$, как показывают вычисления на компьютере, имеет вид:
\begin{align}
\mathcal{R}_3(\mu_1, \mu_2,\mu_3) = (\mu_1\mu_2\mu_3 - 1)(\mu_1\mu_2-1)(\mu_1\mu_3-1)(\mu_2\mu_3-1)\cdot \nonumber \\
\cdot (\mu_1-1)(\mu_2-1)(\mu_3-1)\Big(4(\mu_1\mu_2\mu_3-1)-(\mu_1-1)(\mu_2-1)(\mu_3-1)\Big). \label{eq:small_P3}
\end{align}
Но поскольку, как нетрудно убедиться, множества нулей многочленов $R^*$ и (\ref{eq:small_P3}) совпадают, то далее имеем дело с многочленом (\ref{eq:small_P3}).

Подставляем многочлен $\mathcal{R}_3$ в формулу (\ref{eq:main_Q}). Получаем:
\begin{align*}
\mathcal{L}_4(\mu_1,\mu_2,\mu_3,\mu_4) = \Lambda_4(\mu_1,\mu_2,\mu_3,\mu_4) \cdot \\
 \cdot \Big(4(\mu_1\mu_2\mu_3-1)-(\mu_1-1)(\mu_2-1)(\mu_3-1)\Big) \cdot \\
 \cdot \Big(4(\mu_1\mu_2\mu_4-1)-(\mu_1-1)(\mu_2-1)(\mu_4-1)\Big) \cdot \\
 \cdot \Big(4(\mu_1\mu_3\mu_4-1)-(\mu_1-1)(\mu_3-1)(\mu_4-1)\Big) \cdot \\
 \cdot \Big(4(\mu_2\mu_3\mu_4-1)-(\mu_2-1)(\mu_3-1)(\mu_4-1)\Big),
\end{align*}
что совпадает с утверждением теоремы.
\end{proof}

\section{Кратные неподвижные точки на действительной прямой}
Заметим, что на протяжении большей части статьи мы не апеллировали к тому факту, что функция $\Delta$ есть отображение Пуанкаре некоторого полицикла. По сути мы искали неподвижные точки функции определённого вида, заданной на некотором интервале. Это позволяет переформулировать результат в терминах функций на действительной прямой.

Пусть $f_i: \mathbb{R}_{>0} \to \mathbb{R}$, $i=1, \ldots, n$, --- $C^r$-гладкие функции на действительной полуоси, $r\geq n$. Полагаем, что функции $f_i$ непрерывно зависят от параметра $\delta$, пробегающего произвольное топологическое пространство $B$ с некоторой фиксированной точкой $0$. Пусть существуют такие положительные числа $\lambda_1, \ldots, \lambda_n$, что имеют место следующие пределы:
\begin{align*}
\lim\limits_{\delta, x \to 0} f_i(x) = 0;
\end{align*}
\begin{align}
\lim\limits_{\delta, x \to 0} x^q \frac{\partial^q}{\partial x^q}\ln|f_i'(x)| = (-1)^{q-1}(q-1)!(\lambda_i-1), \quad q = 1, \ldots, r-1. \label{eq:line_limits}
\end{align}
Другими словами, функции $f_i$ вместе со своими первыми $r$ производными ведут себя как степенные функции с показателем $\lambda_i$ (подробнее см. параграф \ref{sec:saddle_limit}).

По аналогии с формулами (\ref{eq:common_view_Fs}), (\ref{eq:limit_Zs}) и (\ref{eq:cycle_Z}) введём обозначения:
\begin{align*}
F_i = f_i \circ \ldots \circ f_0, \quad f_0 = \mathrm{id}, \quad i=0, \ldots, n; 
\end{align*}
\begin{align*}
Z_i = \frac{F_{i-1}'}{F_{i-1}}, \quad i=1,\ldots,n.
\end{align*}
\begin{align*}
\mathcal{Z}: (\delta, x) \mapsto \big(Z_1(\delta, x) : \ldots : Z_n(\delta, x)\big), 
\end{align*}

Рассмотрим следующую функцию:
\begin{align*}
\Delta(x) = f_n \circ \ldots \circ f_1(x).
\end{align*}
Тогда справедливы следующие теоремы:
\begin{theorem}\label{th:line_main}
Существует такой ненулевой многочлен $\mathcal{L}_n \in $ $ \mathbb{Z}[\lambda_1, \ldots, \lambda_n]$, что для любых чисел $\lambda_1, \ldots, \lambda_n$, удовлетворяющих неравенству $\mathcal{L}_n(\lambda_1, \ldots, \lambda_n) \neq 0$, кратность любой близкой к нулю (при $\delta \to 0$) неподвижной точки отображения $\Delta$ не превосходит $n$.
\end{theorem}

\begin{theorem}
Пусть $\mathcal{F}$ --- множество всех пар $(\delta, x)$, соответствующих неподвижным точкам функции $\Delta$ кратности $m+2\leq r$. Тогда любая предельная при $\delta, x \to 0$ точка $z$ функции $\mathcal{Z}\big|_{\mathcal{F}}$ удовлетворяет системе:
\begin{align*}
Q_{nl}(z) = 0, \quad l=1, \ldots, m,
\end{align*}
где многочлены $Q_{nl}$ задаются формулами (\ref{eq:reccurent_l1} - \ref{eq:reccurent_operator_D}).
\end{theorem}

\begin{theorem}
При $n=1,2,3,4$ следующие многочлены  удовлетворяют требованиям теоремы \ref{th:line_main}:
\begin{enumerate}
\item $\mathcal{L}_1(\lambda_1) = \Lambda_1(\lambda_1)$;
\item $\mathcal{L}_2(\lambda_1, \lambda_2) = \Lambda_2(\lambda_1, \lambda_2)$;
\item $\mathcal{L}_3(\lambda_1, \lambda_2, \lambda_3) = \Lambda_3(\lambda_1, \lambda_2, \lambda_3)$;
\item $\mathcal{L}_4(\lambda_1, \lambda_2, \lambda_3, \lambda_4) = \Lambda_4(\lambda_1, \lambda_2, \lambda_3, \lambda_4)\cdot $

$\cdot M(\lambda_1,\lambda_2,\lambda_3) M(\lambda_1,\lambda_2,\lambda_4) M(\lambda_1,\lambda_3,\lambda_4) M(\lambda_2,\lambda_3,\lambda_4)$,
\end{enumerate}
где многочлены $\Lambda_n$, $M$ те же, что и в теореме \ref{th:small}.
\end{theorem}

Доказываются данные теоремы дословным повторением доказательств теоремы \ref{th:main}, предложения \ref{prop:eqv_Q_common} и теоремы \ref{th:small} соответственно. Разница состоит лишь в том, что вместо леммы \ref{lmm:saddle_limit} фигурирует данное по условию соотношние (\ref{eq:line_limits}). Более того, все три теоремы справедливы для конечно-гладких функций $f_i$, потому что единственное место в доказательстве основных теорем данной статьи, где используется бесконечная гладкость векторных полей, --- это та самая лемма \ref{lmm:saddle_limit}, которую заменяет формула (\ref{eq:line_limits}).

\section{Открытые вопросы}
Мы начали статью с описания имеющихся на сегодняшний день результатов, касающихся оценки цикличности полициклов. Может ли изучение кратных предельных циклов помочь в оценке цикличности? Да. Но лишь в оценке снизу. Это можно сформулировать в виде следующих двух гипотез.

Пусть $\gamma$ --- гиперболический полицикл поля $v_0$ на двумерном ориентируемом многообразии, образованный $n$ (возможно, совпадающими) сёдлами с характеристическими числами $\lambda_1, \ldots, \lambda_n$.

\begin{conjecture}
Существует такое открытое подмножество $ U \subset \mathbb{R}^n_{>0}$, что для любого набора характеристических чисел $(\lambda_1, \ldots, \lambda_n) \in U$ в типичном $n$-параметрическом семействе, возмущающем полицикл $\gamma$, рождается предельный цикл кратности $n$ ($n$ предельных циклов).
\end{conjecture}

\begin{conjecture}\label{conj:reflect}
Существует такое открытое (в индуцированной топологии) подмножество $W$ поверхности $\{(\lambda_1, \ldots, \lambda_n) \in \mathbb{R}^n_{>0} | \lambda_1\ldots\lambda_n = 1\}$, что для любого набора характеристических чисел $(\lambda_1, \ldots, \lambda_n) \in W$ в типичном $(n+1)$-параметрическом семействе, возмущающем полицикл $\gamma$, рождается предельный цикл кратности $n+1$ ($n+1$ предельный цикл).
\end{conjecture}

В обеих гипотезах типичнось подразумевает, что семейство размыкает каждую сепаратрисную связку с ненулевой скоростью. Вполне вероятно, что доказательство этих двух гипотез станет основным сюжетом следующей статьи.

\end{document}